\documentclass[a4paper,10pt]{article}

\usepackage{amsmath}
\usepackage{amsfonts}
\usepackage{amssymb}
\usepackage{amsthm}
\usepackage{srcltx}
\usepackage{tabularx}
\usepackage{textcomp}
\usepackage{color}
\usepackage[dvips]{graphicx}
\usepackage[english]{babel}
\usepackage{fontenc}
\usepackage[colorlinks=true,urlcolor=blue]{hyperref}

\newtheorem {prp} {Proposition}[section]
\newtheorem {theo} {Theorem}[section]

\newtheorem {rmk} {Remark}[section]
\newtheorem {ex} {Example}[section]

\newcommand{\ds}{\displaystyle}
\newcommand{\E}{\mathcal{E}}

\newcommand{\B}{\mathcal{B}}
\newcommand{\N}{\mathbb{N}}

\newcommand{\R}{\mathbb{R}}

\date{\today}
\title{Nonlinear Synchronization on Connected Undirected Networks}
\author{S. Orange$^*$ and N. Verdi\`ere\footnote{LMAH (Laboratoire de
Math\'ematiques Appliqu\'ees du Havre), Universit\'e du Havre, 25 rue Philippe Lebon, BP
540,
76058 Le Havre, France. {\it Sebastien.Orange@univ-lehavre.fr, Nathalie.Verdiere@univ-lehavre.fr}}}

\begin{document}
\maketitle
\begin{abstract}
This paper gives sufficient conditions for having complete synchronization of
oscillators in connected undirected networks. The considered oscillators are not
necessarily identical and the synchronization terms can be nonlinear. An important problem
about oscillators networks is to determine conditions for having complete synchronization
that is the stability of the synchronous state. The synchronization study requires to take
 into account the graph topology. In this paper, we extend some results to non linear
cases and we give an existence condition of trajectories. Sufficient conditions given in
this paper are based
on the study of a Lyapunov function and the use of a pseudometric which enables us to
link
network dynamics and graph theory. Applications of these results are presented. 
\end{abstract}


\section{Introduction}

The study of the dynamics of coupled nonlinear dynamical systems are the subject of a
growing interest in various communities like in theoretical physic, in information
technology or in neuronal biology. The literature on this topic shows different kinds of
synchronization (see~\cite{pikovsky2003synchronization}). Classically, two
coupled limit-cycle are said synchronized when their time evolution is
periodic with the same period and perhaps the same
phase. From the discover of synchronization of chaotic systems
(see~\cite{afraimovich1986stochastically,fujisaka1983stability,pecora}), the word
synchronization recovered
different meanings such as having identical or functional related solutions, eventually
with a
delay. The definition has also been modulated by
considering strong forms like complete, cluster
form or weaker forms
like phase and lag synchronization (see~~\cite{rosenblum1997phase}).

An important question about synchronization of a network of oscillators is to determine
the stability of the synchronisation state. This question leads to consider some
properties of networks and state vectors of oscillators (see, for
example,~\cite{belykh2004connection,wu2002synchronization,wu2005synchronization,
wu1995synchronization,zhou2008pinning}). For this purpose, two methods are proposed in
the literature. The first one called \textit{master
stability function} is
based on the
computation of a Lyapunov exponent and the eigenvalues of the connectivity
matrix~\cite{pecora1998master}. However, this method
is adapted when the coupling terms are linear and the computation of eigenvalues can
become a difficult task. A second proposed method is the \textit{connection graph
stability
method} (see~\cite{belykh2004connection}). It links the study
of a Lyapunov
function and the graph topology. This productive method has been extended to unbalance and
undirected graph (see~\cite{belykh2006synchronization, belykh2005synchronization}).

The results presented in this paper generalize some results
of~\cite{belykh2004connection} to the non linear synchronization
case. For this, we introduce a notion of
pseudometric in the graph. The determination of the sign of the
Lyapunov function derivative requires two
steps. The first one is to use assumptions allowing comparisons between oscillators and
synchronization terms. The second step consists in using pseudometrics which
enable us to use some graph properties. For the complete synchronization, we present two
results. The first one gives a condition on synchronization strength for having a global
synchronization of oscillators. The second result is a local versus of the first one, that
is when the
oscillators are closed to the synchronization variety. In these two cases, we
give sufficient conditions that insure existence of trajectories.

This paper is organized as follows. The problem statements are
presented in Section~\ref{sec_ps}. First, we precise the kind of systems and
the kind of synchronizations considered. Then, we recall the
definition and some properties of pseudometrics defined on a graph. In
Section~\ref{sec_hst},
after precising the assumptions on the synchronization term, main results, that is
conditions for having complete synchronization of the system of oscillators, are
presented. These results are applied in Section~\ref{sec_app}.
   
\section{Problem statements}

Thereafter, $Y^T$ is the transpose of the vector $Y= (Y^1, \ldots, Y^m ) \in \R^m$.

\label{sec_ps}
\subsection{Systems and synchronizations considered}
Let $G$ be a connected undirected graph and $n$ its number of vertex. The graph $G$
describes the set of interactions between the oscillators. We
denote by $\E$ the set of its edges. If $G$ contains an undirected
edge from a vertex $i$ to a vertex $j$, we denote it by $(i,j)$. 

The considered dynamical systems are defined by the following system of equations:
\begin{equation}\label{eqn0}
\left\{
\begin{array}{l}
\ds \dot{X}_1 = F_1(X_1,t) -\epsilon \sum_{(1,j)\in \E}  h(X_1,X_j),\\
\phantom{\dot{X}_1 ~\,}\vdots\\
\ds \dot{X}_n = F_n(X_n,t) -\epsilon \sum_{(n,j) \in \E} h(X_n,X_j),\\
\end{array}\right.
\end{equation}

where
\begin{itemize}
 \item $X_i = ( X_i^1, \ldots, X_i^d )^T $ is the vector composed of the $ d$
coordinates of the $i$-th oscillator,
\item $F_i = ( F_i^1, \ldots, F_i^d)^T$ is the vectorial function defining one oscillator,
\item $h= ( h^1, \ldots, h^d)^T$ is the synchronization
function which defines the vector coupling between oscillators,
\item the real parameter $\epsilon $ corresponds to the synchronization strength
\end{itemize}

\noindent Recall that, for a given initial state of the set of oscillators $(X_1(0),
\,
X_2(0), \cdots X_n(0))^T\,, $ system~(\ref{eqn0}) synchronizes completely if, for all
$(i,j) \in\text{\textlbrackdbl} 1,n \text{\textrbrackdbl} $, $$\|X_i(t)-X_j(t) \|
\xrightarrow[t \rightarrow +\infty]{} 0 \, .$$
This means that the vector $( X_1, \ldots, X_n ) $ approaches the synchronization
manifold
defined by $X_1(t)=X_2(t)= \cdots = X_n(t)$. In particular, this implies that the
oscillators have the same asymptotic behavior (such as chaotic trajectories, stable and
periodic solutions). The complete synchronization of all oscillators can occur whatever
their initial states are, in this case, the synchronization is said global;
otherwise it is said local.

In this paper, we focus naturally on the differences $\Delta_{i,j} = X_i^T -X_j^T$ and
therefore
on the
vector $$\Delta=(\Delta_{1,2},\,\cdots ,\,\Delta_{1,n},\,\Delta_{2,3},\,\cdots
,\,\Delta_{2,n},\, \cdots
,\,\Delta_{n-1,n})^T\label{vect_diff}\;.$$ Thus, proving the complete synchronization of
system~(\ref{eqn0})
is
equivalent to prove that $\|\Delta (t)\| \xrightarrow[t \rightarrow +\infty]{} 0 \, .$

\subsection{Quasimetrics defined on a graph}

In the following, we consider
{\it pseudometric} verifying the
$\rho$-relaxed triangle inequality for a positive real $\rho$, that is
an application $\varphi : D \times D \rightarrow \R^+$, where $D$ is an non empty set,
satisfying the
following three axioms:
\begin{itemize}
 \item $\varphi(z_1, z_1) = 0$;
 \item $\varphi(z_1, z_2) = \varphi(z_2, z_1)$ (symmetry property);
 \item $\varphi(z_1, z_3) \leq \rho \,(\varphi(z_1, z_2) + \varphi(z_2, z_3))$
($\rho$-relaxed triangle inequality).
\end{itemize}
Remark that any classical metric is such a pseudometric with $\rho = 1$. \\

Let $\varphi$ be a pseudometric on a set $D$. Let's set, for all $m
\in \N^*$, $\rho(m)$ the smallest
real such that
\begin{equation}
\varphi(z_1,z_{m+1})
\leq \rho(m)
\,\left[ \varphi(z_1, z_2) + \cdots + \varphi(z_{m},z_{m+1})\right]\,.\label{rten}
                                                        \end{equation} Note
that $\rho(1)=1$.\\

In the following examples, expressions of $\rho(m)$ appearing in inequalities~(\ref{rten}) are direct consequences of
the convexity
of functions $ x \rightarrow (x^2)^\alpha$ and $ x \rightarrow x^2\, e^{1-|x|}$.

\begin{ex} \label{expseudo}
\begin{enumerate}
 \item The application $\varphi_\alpha : \R^2 \times \R^2 \rightarrow \R^+$
defined by $$\varphi_\alpha\left(\left(
\begin{array}{l}
x_1\\y_1
\end{array}\right),
\left(
\begin{array}{l}
x_2\\y_2
\end{array}\right)\right) =
\left((x_1-x_2)^2\right)^\alpha$$ with $\alpha \geq 1/2$ is a pseudometric for which $
\rho(m)= m^{2\alpha-1}$.
 \item Let $D$ be the closed ball of center $0$ and radius $2-\sqrt{2}$. The
application
$\varphi : D\times D \rightarrow \R^+$ defined by
$$\varphi\left(\left(
\begin{array}{l}
x_1\\y_1\\z_1
\end{array}\right),
\left(
\begin{array}{l}
x_2\\y_2\\z_2
\end{array}\right)\right) =
(x_1-x_2)^2 e^{1-|x_1-x_2|}$$ is a pseudometric for which $ \rho(m)=m$.
\end{enumerate}
\end{ex}

We have the following properties. 

\begin{prp} \label{prop_pseudom}
\begin{enumerate}
\item The sequence of reals $(\rho(m))_{m\geq 1 }$ is increasing.
\item For all $ m \in  \N^*$, we have $\rho(m) \leq \rho^{m-1}$ (see~\cite{MR2481970}).
\item Let $\varphi_1$ and $\varphi_2$ be two pseudometrics on $D$ and $\rho_1(m)$ and
$\rho_2(m)$ be the smallest respective reals verifying~(\ref{rten}). For all
$\alpha> 0$ and
$\beta>0$, the application $\alpha\, \varphi_1 +\beta
\,\varphi_2$
is a pseudometric on $D$ satisfying $\rho(m)=Max\{\rho_1(m),\rho_2(m)\}$.
\end{enumerate}
\end{prp}

We now apply pseudometrics to networks of oscillators. Recall that a state vector $z_i$ of
an oscillator is associated to $i$-th vertex of $G$. Let's consider a
pseudometric $\varphi$ defined on the set of state vectors of oscillators. This
pseudometric enables one to define the {\it pseudolength} $\varphi(z_i,z_j) $ between
vertices $i$ and $j$ and also the pseudolength $\varphi(z_{i_1}, z_{i_2}) + \cdots +
\varphi(z_{i_{m-1}},z_{i_m})$ of any path $P_{i,j}=(i=i_1,i_2,\cdots,i_m=j)$ from vertex
$i$ to vertex $j$.

In the following proposition, we bound, up to a
multiplicative constant $C(G)$, the sum of pseudolengths between any two oscillators by
the sum of pseudolengths of paths joining any two oscillators. This constant
plays an important role in Theorems~\ref{theo} and~\ref{theo2} since the synchronization
strenght $\epsilon$ appearing in these theorems is proportionnal to this constant.

\begin{prp}
Let $G$ be a connected graph, $\E$ be the set of its edges and $\varphi$ be a pseudometric
on a
set $D$. For any vertex $i$, let $z_i \in D$ be a vector associated to vertex~$i$. There
exists a constant $C$ depending only on $G$ so that we have  
\begin{equation}
                 \sum_{i,j}\varphi(z_i,z_j) \leq C \sum_{(i,j)\in \E} 
\varphi(z_i,z_j)\,.\label{borne}
                \end{equation}
Moreover, the smallest real $C$ satisfying (\ref{borne}), $C(G)$, is bounded by
\begin{equation}\dfrac{n(n-1)}{2}\delta(G)\,\rho(\delta(G))\,,\label{borne_gene}
                \end{equation}
where $\delta(G)$ is the diameter of $G$.
\end{prp}

\begin{proof}
Let $i$ and $j$ be two vertices of $G$ and let's denote
$$P_{i,j}=(i=i_1,i_2,\cdots,i_{s+1}=j)$$ a path of $G$ from
the vertex $i$ to vertex $j$ (recall that $G$ is connected). Since $\varphi$ is a
pseudometric on $D$, we have
$\varphi(z_i,z_j)\leq \rho(s) \sum_{\ell=1}^{s}
\varphi(z_{i_\ell},z_{i_{\ell+1}})\,. $

The path $P_{i,j}$ can be chosen so that $s \leq \delta(G)$. Suppose that this choice is
done for any vertices $i$ and
$j$; since the sequence $(\rho(n))_{n \in \N^*}$ is increasing, we have $\rho(s)
\leq \rho(\delta(G))$. Consequently, for any vertices $i$ and $j$, we
have $\varphi(z_i,z_j)\leq \rho(\delta(G))\;\delta(G)\; Max\left(\{\varphi(z_i,z_j)\mid
(i,j) \in \E\}\right) $ which implies the result.
\end{proof}
In~Theorem~\ref{theo}, we need to determine the lowest bound $C(G)$ of the set of reals
$C$ satisfying inequality~(\ref{borne}). The bound~(\ref{borne_gene}) of $C(G)$ may not
lead to a good estimation of $C(G)$ for a particular graph; nevertheless, this bound is
valid for any graph with $n$ vertices.\\ 

In the case of a pseudometric satisfying the classical triangle
inequality, i.e. when $\rho(n)=n$ for all $n\in \N^*$, a method taking $G$ as input and
returning a bound of $C(G)$ is proposed in~\cite{belykh2005synchronization}. Its two main
steps are:
\begin{enumerate}
 \item for all $(i,j) $ with $i>j$, choose a path $P_{i,j}$; this path is usually chosen 
with minimal length (number of edges in the path);
 \item for each edge $e$ of the connection graph, determine the sum $B(e)$ of the lengths
of all chosen paths $P_{i,j}$ containing $e$. A bound for $C(G)$ is then $Max\{B(e) :
e \in \E \}$.
\end{enumerate}

For each choice of paths, these two steps return a bound for $C(G)$. Clearly, the number
of possible paths is huge but computations of bounds for $C(G)$ are possible
since most of these choices are suboptimal. Up to a slight modification of the first
step, this method can be applied here: its consists in
considering, for all path $P_{i,j} $, the pseudolength $\rho(|P_{i,j}|)$ instead of its
length
$|P_{i,j}|$. 

\begin{rmk} \label{rem_cg} In the case of pseudometrics
$\varphi$ satisfying $\rho(m) = m$, explicit bounds of $C(G)$ for specific graphs and the
method proposed
in~\cite{belykh2004connection,belykh2005synchronization} for computing $C(G)$ from $G$ can
be directly used. This is
the case of the
second function in Example~\ref{expseudo}. 
\end{rmk}

\section{Complete synchronizations}
\label{sec_hst}
\subsection{Hypothesis}\label{sec:hyp}
Afterwards, two cases are considered. The first one is the global complete
synchronization for which oscillators $X_1,\ldots,X_n$ lies in $D = \R^d$. The second one
is the complete
synchronization for which oscillators are in a neighborhood $D$
of the variety $X_1=X_2 = \cdots = X_n$.

Thereafter, we will suppose the following
assumptions on system~(\ref{eqn0}).
\begin{itemize}
\item For all $ (i,j) \in \E$,
there exist some non negative reals $a_1,\, \ldots,\, a_d$  such
that
\begin{equation}
\forall (X_i,X_j)\in D, \;\varphi(X_i,X_j) = \sum_{k=1}^d
a_k(X_i^k-X_j^k)h^k(X_i,X_j)\label{not_varphi} 
\end{equation} are pseudometrics where $h=(h^1,\ldots,h^d)^T$ is the synchronization
function.
\item For all $(i,j)\in \text{\textlbrackdbl} 1,n
\text{\textrbrackdbl}^2 $ and, for all $t\geq t_0$ where $t_0\in \R$,
\begin{equation}\forall (X_i,X_j)\in D, \;
{\sum_{k=1}^d  a_k(X_i^k-X_j^k)\left(F_i^k(X_i,t)-F_j^k(X_j,t)\right)} \leq
{\varphi(X_i,X_j) 
}\,. \label{hyp0} 
\end{equation}
\item  For all $(i,j)\in \text{\textlbrackdbl} 1,n
\text{\textrbrackdbl}^2 $, $\forall (X_i,X_j)\in D, \;$
\begin{small}\begin{equation}
\begin{array}{c}
\varphi(X_i,X_j)=0 \text{ and/or } \sum_{k=1}^d 
a_k(X_i^k-X_j^k)\left(F_i^k(X_i,t)-F_j^k(X_j,t)\right)=0 \hfill~ \\
~\hfill
\Rightarrow
(X_i=X_j)\,. 
\label{SD} 
\end{array} 
\end{equation}              \end{small}
\end{itemize}
\begin{rmk}
\begin{enumerate}
 \item Notice that hypothesis~(\ref{not_varphi}) implies that,
\begin{equation}\forall
(i,j) \in \E,\; \forall (X_i,X_j) \in
D, \; h(X_i,X_j) =- h(X_j,X_i)\,\text{(antisymmetry)}.\label{hyp1} 
\end{equation}
 \item The assumption~(\ref{SD}) is necessary for proving the complete synchronisation of
system~(\ref{eqn0}) in Theorems~\ref{theo} and~\ref{theo2}. The
condition $\varphi(X_i,X_j)=0$ in this assumption is not always sufficient when it does
not imply equalities of all the components of oscillators. In this case, the
second condition is
necessary for proving the complete synchronization.
\end{enumerate}
\end{rmk}

For practical cases, a first problem is to prove the existence of trajectories of
system~(\ref{eqn0}) for a sufficient large $t$. For this goal, the following proposition
enables us to link existence of trajectories between synchronized and non synchronized
systems.

\begin{prp}\label{existence} For all $(i,j) \in \text{\textlbrackdbl}1,n
\text{\textrbrackdbl}^2$, suppose that assumptions~(\ref{not_varphi}), (\ref{hyp0})
and (\ref{SD}) are satisfied and that, for all $t\geq t_0$, $$X_i^T F_i(X_i,t)\leq
\Psi(\mid\mid X_i\mid \mid)$$ where $\Psi$ satifies
the conditions \begin{center}
 $\ds \int_{s=s_0}^{+\infty}\dfrac{ds}{\Psi(t)}= +\infty$ and $\Psi(s) >0$ for
all $s\geq s_0 \geq 0$.\end{center}

Then, the Cauchy's problem defined by system~(\ref{eqn0}) and an initial condition 
$\left( \begin{array}{c}
      X_1(t_0)\\ \vdots\\X_n(t_0)
     \end{array}\right)
\in \R^{nd}$ has a solution on the
complete semi-axis $[t_0; +\infty)$ .\\
\end{prp}

\begin{proof}Let's set $X =\left( \begin{array}{c}
      X_1\\ \vdots\\X_n
     \end{array}\right)
\in \R^{nd}$ and
$F(X,t)=\left( \begin{array}{c}
     F_1({X}_1,t)\\ \vdots\\F_n({X}_n,t)
     \end{array}\right)\in \R^{nd}$.
In a first step, we prove that there exists a real $\beta$ such that the following
inequality between the scalar products holds:\begin{equation}
\label{ineq}
                                        X^T\dot{X} \leq \beta X^TF(X,t).
                                       \end{equation}
 
For this, we consider the $dn \times dn$ diagonal matrix $ M =
Diag(a_1,\ldots a_d, \ldots ,a_1,\ldots a_d).$ We have:
\begin{equation*}
\begin{array}{rcl}
X^TM \dot{X}& =&\ds \sum_{i=1}^n 
\sum_{k=1}^d a_k  X_i^k F_i^k(X_i,t) -\epsilon \sum_{i=1}^n\sum_{k=1}^d  a_k \sum_{\{j |
(i,j)\in \E\}}  X_i^k h^k(X_i,X_j)\\
&=&\ds X^TM F(X,t) -\epsilon\sum_{k=1}^d \sum_{(i,j)\in
\E} a_k  X_i^k h^k(X_i,X_j)\\
\end{array} 
\end{equation*}
and, since to any edge $(i,j)\in \E$ corresponds the edge $(j,i)\in \E$, we obtain
\begin{equation*}
\begin{array}{rcl}
X^T M \dot{X}& =& \ds X^T M F(X,t)-\frac{\epsilon}{2}\sum_{k=1}^d a_k \sum_{(i,j)\in
\E}  X_i^k h^k(X_i,X_j) +  X_j^k h^k(X_j,X_i) \\

& =& \ds X^T M F(X,t)-\frac{\epsilon}{2}\sum_{k=1}^d
a_k\sum_{(i,j)\in \E}  (X_i^k-X_j^k) h^k(X_i,X_j)\text{ (see~equality~(\ref{hyp1}))} \\
& =& \ds X^T M F(X,t)-\frac{\epsilon}{2}\sum_{(i,j)\in \E} \varphi(X_i,X_j) \\
& \leq & \ds X^T M F(X,t). \text{
(see~assumption~(\ref{not_varphi}))} \\
\end{array} 
\end{equation*}
Inequality~(\ref{ineq}) is then a direct consequence of the fact that the reals $a_i$ are
non negative. \\

If the conditions of the proposition are verified, inequality~(\ref{ineq}) shows that
we have, for all $t\geq t_0$,
$$X^T\dot{X} \leq \widetilde{\Psi}(\mid\mid X \mid\mid)$$

where $\widetilde{\Psi}$ is a application satifying
the conditions \begin{center}
 $\ds \int_{s=s_0}^{+\infty}\dfrac{ds}{\widetilde{\Psi}(t)}= +\infty$ and
$\widetilde{\Psi}(s) >0$
for all $s\geq s_0 \geq 0$.\end{center}

Thus, system~(\ref{eqn0}) satisfies the conditions of
Wintner's theorem (\cite{wintner1945non}) and, consequently, solutions of
system~(\ref{eqn0})
are defined for any $t\geq t_0$.

\end{proof}

\subsection{Global synchronization}

\begin{theo}\label{theo} Suppose that the assumptions done in Section~\ref{sec:hyp} are
satisfied for $D=(\R^d)^2$.
If  $\epsilon > \dfrac{ C_G }{2n}$, where $C_G$
is the optimal bound such
that
inequa\-lity~(\ref{borne}) holds, then system~(\ref{eqn0}) synchronizes completely.
\end{theo}

\begin{proof} In order to show this result, we will apply the second method of Lyapunov.
Let's consider the Lyapunov candidate function:

$$V=\dfrac{1}{2} \sum_{k=1}^d\sum_{i\leq j} a_k(X^k_i-X^k_j)^2 \,.$$
Clearly, this function is non negative if $\Delta\neq \overrightarrow{0}$ and equal to
$0$
iff $\Delta = \overrightarrow{0}$ that is when the system~(\ref{eqn0}) is
synchronized. 

The derivative of $V$ gives:
\begin{equation*}
\begin{array}{rcl}
\ds \dot{V}&=&\ds\sum_{k=1}^d a_k\dfrac{1}{2} \sum_{i=1}^n\dfrac{\partial V}{\partial
X^k_i} \; \dot{X}^k_i\\
&=&\ds \sum_{k=1}^d a_k\sum_{i=1}^n (n  X_i^k - \sum_{j=1}^n X^k_j) \dot{X}^k_i  \\

&=&\ds \sum_{k=1}^d a_k \left( n\sum_{i=1}^n X^k_i\dot{X}^k_i - \sum_{j=1}^n X^k_j
\sum_{i=1}^n\dot{X}^k_i \right) \\

&=&\ds \sum_{k=1}^d a_k \left[ n \left(\sum_{i=1}^n X_i^k F_i^k(X_i,t)
-\epsilon \sum_{i=1}^n\sum_{\{j |
(i,j)\in
\E\}}X^k_i \,h^k(X_i,X_j)\right)\right.\\

& &\ds\left. \qquad - \sum_{j=1}^n X^k_j \left(\sum_{i=1}^n  F_i^k(X_i,t)  -\epsilon 
\sum_{i=1}^n\sum_{\{j | (i,j)\in
\E\}}h^k(X_i,X_j) \right)\right]\\

&=&\ds \sum_{k=1}^d a_k\left[ \sum_{i=1}^n \left( n X^k_i -\sum_{j=1}^n X^k_j
\right)F_i^k(X_i,t) \right. \\

& &\ds\left.\qquad -n\epsilon  \sum_{ (i,j)\in
\E }X^k_ih^k(X_i,X_j)+ \epsilon \left(\sum_{j=1}^n X^k_j \right) \sum_{ (i,j)\in
\E }h^k(X_i,X_j)\right] \\
&=& \ds\sum_{k=1}^d a_k\left[ \sum_{(i,j) \in \text{\textlbrackdbl} 1,n
\text{\textrbrackdbl}}^n \left(X^k_i -X^k_j \right)F_i^k(X_i,t) \right. \\

& &\ds\left.\qquad -n \epsilon \sum_{ (i,j)\in
\E }X^k_ih^k(X_i,X_j)+  \epsilon  \left(   \sum_{j=1}^n X^k_j  \right)   \sum_{
(i,j)\in
\E }h^k(X_i,X_j) \right]\,.\\
\end{array} 
\end{equation*}
Since  each edge $(i,j) \in \E$ corresponds to
an edge $(j,i)$ and using equality~(\ref{hyp1}), we have, for all $k\in
\text{\textlbrackdbl} 1,n \text{\textrbrackdbl}$,
\begin{equation*}
\begin{array}{rcl}
\ds 2\sum_{ (i,j)\in \E }h^k(X_i,X_j) 
&=&\ds \sum_{ (i,j)\in\E }h^k(X_i,X_j) + \sum_{ (i,j)\in\E
}h^k(X_j,X_i)\\ 
&=&\ds  \sum_{ (i,j)\in\E }h^k(X_i,X_j) + \sum_{ (i,j)\in\E
}-h^k(X_i,X_j)\\ 
&=&0 \,
\end{array} 
\end{equation*}
and 
\begin{equation*}
\begin{array}{rcl}
\ds 2\sum_{k=1}^d a_k\sum_{ (i,j)\in \E }X^k_ih^k(X_i,X_j) 
&=&\ds\sum_{k=1}^d a_k\left[ \sum_{ (i,j)\in\E }X^k_ih^k(X_i,X_j) + \sum_{ (i,j)\in\E
}X^k_jh^k(X_j,X_i)\right]\\ 
&=&\ds\sum_{k=1}^d a_k \left[ \sum_{ (i,j)\in\E }X^k_ih^k(X_i,X_j) + \sum_{ (i,j)\in\E
}-X^k_jh^k(X_i,X_j)\right]\\ 
&=&\ds \sum_{ (i,j)\in\E }\varphi(X_i,X_j) \,\text{(see~\ref{not_varphi})}.
\end{array} 
\end{equation*}
Moreover, we have
\begin{equation*}
\begin{array}{rcl}
\ds 2\sum_{i,j}(X^k_i-X^k_j)F_i^k(X_i,t) 
&=&\ds \sum_{ i,j }(X^k_i-X^k_j)F_i^k(X_i,t)  + \sum_{ i,j
}(X^k_j-X^k_i)F_j^k(X_j,t) \\ 
&=&\ds \sum_{ i,j }(X^k_i-X^k_j)(F_i^k(X_i,t)-F_j^k(X_j,t))\,.
\end{array} 
\end{equation*}
These three equalities gives
\begin{equation}\label{vprime}
\ds \dot{V}=\ds \ds\sum_{i,j} \sum_{k=1}^d  \frac{a_k}{2}(X^k_i-X^k_j)
\left({F_i^k(X_i,t)-F_j^k(X_j,t)}\right)-n\epsilon \sum_{
(i,j)\in \E }\varphi(X_i,X_j)
\end{equation}
With assumption~(\ref{hyp0}) and inequality~(\ref{borne}), we obtain 
\begin{equation*}
\begin{array}{rcl}
\ds \dot{V}
&\leq& \ds \frac{1}{2}\sum_{i,j}  \varphi(X_i,X_j)
-n\epsilon \sum_{
(i,j)\in \E }\varphi(X_i,X_j) \\
&\leq& \ds \left(   \frac{C_G}{2} - n\epsilon 
\right)  \sum_{(i,j)\in \E} \varphi(X_i,X_j) \\
\end{array} 
\end{equation*}
Since $ \varphi$ is a pseudometric the right factor of this last
expression is non negative. Therefore,  if $\epsilon > \dfrac{ C_G }{2n}$  then $
\dot{V} \leq 0$. To prove that $\dot{V}$ is negative definite, it remains to show that
if $\dot{V}=0$ then $X_1=X_2=\cdots = X_n$.
Suppose that $\dot{V}=0$. Since $ \left(\frac{C_G}{2} - n\epsilon 
\right)<0$, the last inequality implies that we have $\varphi(X_i,X_j)=0$ for all
$(i,j)\in \E$. From equality~(\ref{vprime}), we obtain $$\sum_{i,j}
\sum_{k=1}^d {a_k}(X^k_i-X^k_j)
\left({F_i^k(X_i,t)-F_j^k(X_j,t)}\right)=0 \,. $$Consequently,
assumption~(\ref{SD}) is satisfied and system~(\ref{eqn0}) synchronizes.
\end{proof}

\subsection{Local synchronization}
Let $H$ be the diagonal matrix $Diag(a_1,\ldots,a_d)$ and $\mathcal{H}=\left(
\begin{array}{cccc}
H&0&\cdots&0\\
0&H&\cdots&0\\                                                        
\vdots&\vdots&\ddots&\vdots\\
0&0&\cdots&H  \end{array}\right)
$ the matrix composed with $\frac{n(n-1)}{2}$ matrices $H$. The application

\begin{equation}\label{defnv}
\begin{array}{lccc}
 \|.\|_V:& \R^{\frac{n(n-1)}{2}d} &\rightarrow &\R^+\\
&X&\rightarrow&\sqrt{\frac{1}{2}X^T\mathcal{H}X}
\end{array}
\end{equation}

 is a norm since $a_1,\ldots,a_d$ are non negative. Let's set $$V(t) = 
\|\Delta(t)\|_V^2 ={\frac{1}{2} \sum_{k=1}^d\sum_{i< j\leq n}
a_k(X_i^k(t)-X_j^k(t))^2}\,.$$ 

\begin{theo}\label{theo2} Let $\B$ the closed ball $\{X \in\R^{\frac{n(n-1)}{2}d}  \mid
 \|X\|_V\leq {r}\}$ where $r$ is a non negative real. Suppose that assumptions of
Section~\ref{sec:hyp} are satisfied when $\Delta$ belongs to the inner
$\stackrel{\circ}{\B}$ of $\B$ and suppose that,
for an instant $t_0$, $\Delta(t_0) \in\; \stackrel{\circ}{\B}$.\\ 
If  $\epsilon > \dfrac{ C_G }{2n}$, where $C_G$ is the optimal bound such
that inequality~(\ref{borne}) holds, then system~(\ref{eqn0}) synchronizes.
\end{theo}

\begin{proof} 

Let's show that if $\Delta(t_0) \in  \stackrel{\circ}{\B}$ then $\forall t > t_0$, $
\Delta(t) \in \B$. 
If $\Delta(t_0) \in  \stackrel{\circ}{\B}$, by definition of $\B$, we have $V(t_0)<
r^2$.
Suppose that
there exists $t_1>t_0$ such that $\Delta(t_1) \notin \B$; by definition of $\B$, we have
$V(t_1)>r^2$. Since $t \rightarrow V(t)$ is continuous, there exists a real $t_2 =
Inf\{t \in [t_0,t_1] | V(t) = r^2\}$. The mean value theorem shows that there
exists $t_3 \in (t_0,t_2)$ such that $V'(t_3) =
\frac{V(t_0)-V(t_2)}{t_0-t_2}>0.$\\
On the other side, since $t_3 <t_2 = Inf\{t \in [t_0,t_1] |
V(t) = r^2\}$, we have $V(t_3)<r^2$ and $\Delta(t_3) \in\; \stackrel{\circ}{\B} $.
Consequently, the hypothesis of
Section~\ref{sec:hyp} are satisfied by $\Delta(t_3)$ and we can
proceed like in the proof of Theorem~\ref{theo} to show that $V'(t_3)\leq 0$. This
brings to a contradiction.

Finally, we have $\forall t \geq t_0$, $ \Delta(t) \in \B$ and the assumptions
of
Section~\ref{sec:hyp} are satisfied for any $ t \geq t_0$. Now, we can proceed like in
the proof of Theorem~\ref{theo} to conclude.
\end{proof}

\section{Applications}
\label{sec_app} In this section, we focus on applications of Theorems~\ref{theo} and
\ref{theo2} in order to have a sufficient condition for global synchronization of two
systems. The fact that solutions of these two
systems are defined on $\R$ is a direct
consequence of Proposition~\ref{existence}. 

\subsection{Global synchronization of a network of neurons}

In this section, we apply Theorem~\ref{theo} to a network of neurons satisfying the
FitzHugh-Nagumo model (See~\cite{hindmarsh1982model}). Recall that the dynamic of a single
neuron is mode\-lised by the
equation
$\dot{X}=F(X)$ where
\begin{itemize}
 \item $X=\left( \begin{array}{c}
      x\\y
     \end{array}\right)$;
 \item $F(X) =\left( \begin{array}{c}
         -x^3+x -y + a \\
     b x-cy-d 
     \end{array}\right)$ for some real parameters $a$, $b$, $c$ and $d$.
\end{itemize}

In the following, we suppose that $b$ is positive. Let's set $G$ the connected
graph describing the interaction between the oscillators, $n$ its number of vertices and
$\E$ the set of its edges. For the synchronization terms, we consider the function $h$
defined by $$\forall(i,j)\in \text{\textlbrackdbl} 1,n
\text{\textrbrackdbl}^2,\; h(X_i,X_j)=\left( \begin{array}{c}
     \alpha(x_i-x_j)+\beta\sqrt[3]{(x_i-x_j)^5}\\
\gamma(y_i-y_j)
     \end{array}\right)$$ with $\alpha \geq 1$, $\beta \geq 0$ and
$\gamma\geq Max\{0,-c\}$. The system of equations for the network of oscillators is then
\begin{equation}\label{eqnFNex1}
\left\{
\begin{array}{l}
\ds\dot{X}_1 = F_1(X_1) -\epsilon \sum_{(1,j)\in
\E}h(X_1,X_j),\\
\phantom{\dot{x}_1 ~\,}\vdots\\
\ds \dot{X}_n = F_n(X_n) -\epsilon \sum_{(n,j) \in
\E}h(X_n,X_j).\\
\end{array}\right. 
\end{equation}

\noindent The three hypothesis of Section~\ref{sec:hyp} are
satisfied with $a_1=1$ and $a_2=1/b$. Indeed,
\begin{enumerate}
 \item assumption~(\ref{SD}) is obvious;
 \item the fact that the application $
\varphi$ corresponding to $h$, explicitly defined by $$\varphi(X_i,X_j) =
\alpha(x_i-x_j)^2+\beta\sqrt[3]{(x_i-x_j)^8}+\gamma/b(y_i-y_j)^2,$$ is a
pseudometric satisfying $\rho(m)=m^{5/3}$ is a consequence of Example~\ref{expseudo} and
Proposition~\ref{prop_pseudom}. Therefore,
assumption~(\ref{not_varphi}) is satisfied;
 \item the following inequalities shows assumption~(\ref{hyp0}), for all $(X_i,X_j)\in
D$, 

 \begin{small}
$\begin{array}{l}
\sum_{k=1}^2 a_k(X_i^k-X_j^k)\left(F_i^k(X_i)-F_j^k(X_j)\right)\hfill~\hfill~ \\
~\hspace{0.21\linewidth}=\left( \begin{array}{c}
             x_i-x_j \\
     y_i-y_j
            \end{array}\right).\left( \begin{array}{c}
             -(x_i^3-x_j^3)+(x_i-x_j) -(y_i-y_j) \\
     (x_i-x_j)-c/b(y_i-y_j) 
            \end{array}\right)\\

~\hspace{0.21\linewidth}= -(x_i-x_j)(x_i^3-x_j^3)+(x_i-x_j)^2 
-c/b(y_i-y_j)^2 \\

~\hspace{0.21\linewidth}\leq
\varphi(X_i,X_j)\,.  
\end{array}$
\end{small}
\end{enumerate}

For any connected graph $G$ with
$n$ vertex, inequality~(\ref{borne}) is verified for the bound of $C(G)$ given by 
$C=\dfrac{n(n-1)}{2}\delta(G)\,\rho(\delta(G))$. Theorem~\ref{theo} shows then that, for
any connected graph $G$ with $n$
vertex, if $\epsilon > \dfrac{(n-1)\,\delta(G)^{8/3}}{4}$ then system~(\ref{eqnFNex1})
synchronizes.

\subsection{Local synchronization of a network of oscillators}

In this section, we apply Theorem~\ref{theo2} to a network of Chua
oscillators. We consider the simplified version suggested by Chua for these oscillators
(see~\cite{matsumoto1984chaotic}): if we set $X =(x,y,z)^T$, the state equation for a
single oscillator is
given by
$\dot{X}= F(X)$ where 
$$F(x,y,z) = \left( \begin{array}{c}a [y-x-f(x)]  \\
         x-y+z   \\
    -b y -cz
            \end{array}\right),\,$$ $a>0$, $b>0$, $c>0$ and $f$ is a piece-wise function
$f(x) = dx + 1/2 (d-e)(|x + 1|-|x-1|)$ with $2d<e$.\\

Since $f$ is a piece-wise function, a real $\delta \geq 0$ bounds the set of
slopes
$\left\{\frac{f(x)-f(y)}{x-y}\,\mid\, 0<|x-y|\leq 1 \right\}$. In the
following, we suppose that:
\begin{enumerate}
 \item the set of
vertex of $G$ is $\E=\{(1;2),(1;3),\, \ldots , \,(1;n)\}$. In other words, we consider a
star configuration of oscillators;
 \item the synchronization function $h$ is given
by $$h((x_i,y_i,z_i),(x_j,y_j,z_j))=\left( \begin{array}{c}
                                            a\delta(x_i-x_j)e^{1-|x_i-x_j|}\\0\\0
                                           \end{array}
\right)\;.$$
\end{enumerate}

\noindent The equation for the $i$-th oscillator of the network is then $$\left(
\begin{array}{l}
\dot{x_i}\\\dot{y_i}\\\dot{z_i}
\end{array}\right) = \left( \begin{array}{c}a [y_i-x_i-f(x_i)]  \\
         x_i-y_i+z_i   \\
    -b y_i-cz_i
            \end{array}\right)+\epsilon \sum_{j \,\mid \, (i,j) \in \E } \left(
\begin{array}{c}a\delta(x_i-x_j)e^{1-|x_i-x_j|}\\0\\0 \end{array}\right) \,.$$
Assumptions of Section~\ref{sec:hyp} have to be verified in order to apply
Theorem~\ref{theo2}. The first one is obvious. For the second and the third one, let's
set $a_1=1/a$, $a_2=1 $ and $a_3=1/b$.

 Let's consider a closed ball $\B=\left\{X
\in\R^{\frac{n(n-1)}{2}d} \mid \|X\|_V\leq {(\sqrt{2}-1)}{\sqrt{a}}\right\}$ where
$\|.\|_V $ is defined by~(\ref{defnv}) and
the norm
$\|.\|_{\tilde{V}}$ given by $$\begin{array}{lccc}
 \|.\|_{\tilde{V}}:& \R^d &\rightarrow &\R^+\\
&Y&\rightarrow&\sqrt{\frac{1}{2}Y^T{H}Y}
\end{array}$$ where $H $ is the diagonal matrix $Diag(a_1,\ldots,a_d)$.  If we have
$\Delta \in \B $ then $\|\Delta_{i,j}\|_{\tilde{V}} < 
{(\sqrt{2}-1)}{\sqrt{a}}$. This implies that $\mid x_i-x_j\mid < 2 - \sqrt{2}$ and,
according to Example~\ref{expseudo},
the application $\varphi$ corresponding to $h$ satisfies assumption~(\ref{not_varphi}). \\

Let's verify assumption~(\ref{hyp0}). We have
 \begin{small}

\noindent
$\begin{array}{l}
\sum_{k=1}^3 a_k(X_i^k-X_j^k)\left(F_i^k(X_i)-F_j^k(X_j)\right)\hfill~ \\
~\hspace{0.09\linewidth}=\left( \begin{array}{c}
            \dfrac{x_i-x_j}{ a} \\
     y_i-y_j \\
\dfrac{z_i-z_j}{b}
            \end{array}\right). \left( \begin{array}{c}a
[(y_i-y_j)-(x_i-x_j)-(f(x_i)-f(x_j))] \\
         (x_i-x_j)-(y_i-y_j)+(z_i-z_j)   \\
    -b (y_i-y_j)-c(z_i-z_j)
            \end{array}\right)\\

~\hspace{0.09\linewidth}= (x_i-x_j)(f(x_i)-f(x_j))-(x_i-x_j)^2
-(y_i-y_j)^2-c/b(z_i-z_j)^2\,.\\
\end{array}$
\end{small}

\noindent By definition of $\delta$, we
have $\begin{array}{l}(x_i-x_j)(f(x_i)-f(x_j))\leq\delta (x_i-x_j)^2
e^{1-|x_i-x_j|}\,.\end{array}$
This shows inequality~(\ref{hyp0}).\\

Moreover, if $\varphi(x_i,x_j)=0$ and $\sum_{k=1}^3
a_k(X_i^k-X_j^k)\left(F_i^k(X_i)-F_j^k(X_j)\right)= 0$ then we have $X_i=X_j$.
Consequently, assumption~(\ref{SD}) holds.\\

Since the induced pseudometric $\varphi$ satisfies $\forall m \in \N^*,\; \rho(m)=m$ (see
Example~\ref{expseudo}), the bound $C_G$ is given explicitly by $2n-3$ (See
Remark~\ref{rem_cg} and~\cite{belykh2004connection}).

Theorem~\ref{theo2} can now be applied : if $\Delta(t_0)
\in\; \stackrel{\circ}{\B}$ for an instant $t_0$ and if $\epsilon > \dfrac{
2n-3}{2n}$ then system~(\ref{eqn0}) synchronizes.
\pagebreak

\section{Conclusion}
In this paper, sufficient conditions for proving complete synchronization of oscillators
in a connected undirected network are presented. The contribution of this paper lies in
the extension of results established in the case of linear synchronization to the non
linear case. For this, we have introduced pseudometrics which enable us to link graph
topology and minimal synchronization strength between oscillators. Under our assumptions,
a criterion proving the existence of trajectories is given. Two results for proving the
complete synchronization are then proposed: the
first one gives a global criterion and the
second one deals with local synchronization, that is when the
trajectories lie in a neighborhood of the synchronization variety.
To illustrate these results, two applications are treated.


\begin{thebibliography}{10}

\bibitem{afraimovich1986stochastically}
VS~Afraimovich, NN~Verichev, and MI~Rabinovich.
\newblock Stochastically synchronized oscillators in dissipative systems.
\newblock {\em Radiophys. Quant. Electron}, 29:795--803, 1986.

\bibitem{belykh2006synchronization}
I.~Belykh, V.~Belykh, and M.~Hasler.
\newblock Synchronization in asymmetrically coupled networks with node balance.
\newblock {\em Chaos: An Interdisciplinary Journal of Nonlinear Science},
  16:015102, 2006.

\bibitem{belykh2005synchronization}
I.~Belykh, M.~Hasler, M.~Lauret, and H.~Nijmeijer.
\newblock Synchronization and graph topology.
\newblock {\em Int. J. Bifurcation and Chaos}, 15(11):3423--3433, 2005.

\bibitem{belykh2004connection}
V.N. Belykh, I.V. Belykh, and M.~Hasler.
\newblock Connection graph stability method for synchronized coupled chaotic
  systems.
\newblock {\em Physica D: nonlinear phenomena}, 195(1-2):159--187, 2004.

\bibitem{fujisaka1983stability}
H.~Fujisaka and T.~Yamada.
\newblock Stability theory of synchronized motion in coupled dynamical systems.
\newblock {\em Prog. Theor. Phys}, 69(1):32--47, 1983.

\bibitem{hindmarsh1982model}
JL~Hindmarsh and RM~Rose.
\newblock A model of the nerve impulse using two first-order differential
  equations.
\newblock 1982.

\bibitem{matsumoto1984chaotic}
T.~Matsumoto.
\newblock A chaotic attractor from chua's circuit.
\newblock {\em Circuits and Systems, IEEE Transactions on}, 31(12):1055--1058,
  1984.

\bibitem{pecora}
L.M. Pecora and T.L. Carroll.
\newblock Synchronization in chaotic systems.
\newblock {\em Physical review letters}, 64(8):821--824, 1990.

\bibitem{pecora1998master}
L.M. Pecora and T.L. Carroll.
\newblock Master stability functions for synchronized coupled systems.
\newblock {\em Physical Review Letters}, 80(10):2109--2112, 1998.

\bibitem{pikovsky2003synchronization}
A.~Pikovsky, M.~Rosenblum, and J.~Kurths.
\newblock {\em Synchronization: A universal concept in nonlinear sciences},
  volume~12.
\newblock Cambridge Univ Pr, 2003.

\bibitem{rosenblum1997phase}
M.G. Rosenblum, A.S. Pikovsky, and J.~Kurths.
\newblock From phase to lag synchronization in coupled chaotic oscillators.
\newblock {\em Physical Review Letters}, 78(22):4193--4196, 1997.

\bibitem{wintner1945non}
A.~Wintner.
\newblock The non-local existence problem of ordinary differential equations.
\newblock {\em American Journal of Mathematics}, 67(2):277--284, 1945.

\bibitem{wu2002synchronization}
C.W. Wu.
\newblock {\em Synchronization in coupled chaotic circuits and systems},
  volume~41.
\newblock World Scientific Pub Co Inc, 2002.

\bibitem{wu2005synchronization}
C.W. Wu.
\newblock Synchronization in networks of nonlinear dynamical systems coupled
  via a directed graph.
\newblock {\em Nonlinearity}, 18:1057, 2005.

\bibitem{wu1995synchronization}
C.W. Wu and L.O. Chua.
\newblock Synchronization in an array of linearly coupled dynamical systems.
\newblock {\em Circuits and Systems I: Fundamental Theory and Applications,
  IEEE Transactions on}, 42(8):430--447, 1995.

\bibitem{MR2481970}
Q. Xia.
\newblock The geodesic problem in quasimetric spaces.
\newblock {\em J. Geom. Anal.}, 19(2):452--479, 2009.

\bibitem{zhou2008pinning}
J.~Zhou, J.~Lu, and J.~L{\"u}.
\newblock Pinning adaptive synchronization of a general complex dynamical
  network.
\newblock {\em Automatica}, 44(4):996--1003, 2008.
\end{thebibliography}
\end{document}